\documentclass[12pt]{amsart}

\usepackage{amsmath}
\usepackage{amscd}
\usepackage{amssymb}
\usepackage{amsfonts}
\usepackage{amsthm}
\usepackage{enumerate}
\usepackage{hyperref}
\usepackage{color}
\usepackage{graphicx}
\theoremstyle{plain}
\newtheorem{thm}{Theorem}[section]

\newtheorem{prop}[thm]{Proposition}
\newtheorem{lem}[thm]{Lemma}
\newtheorem{cor}[thm]{Corollary}

\theoremstyle{definition}
\newtheorem{dfn}[thm]{Definition}

\newtheorem{exmp}[thm]{Example}

\newtheorem{rem}[thm]{Remark}

\newtheorem{dfns-rems}[thm]{Definitions and Remarks}
\newtheorem{notas-rems}[thm]{Notations and Remarks}
\newtheorem{exmps-rems}[thm]{Examples and Remarks}

\DeclareMathOperator{\min-match}{min-match}
\DeclareMathOperator{\ind-match}{ind-match}

\begin{document}

% ------------------------------------------------------------------------

\title[On the regularity of edge ideal of graphs]{On the regularity of edge ideal of graphs}

% ------------------------------------------------------------------------

\author[S. A. Seyed Fakhari]{S. A. Seyed Fakhari}

\address{S. A. Seyed Fakhari, School of Mathematics, Statistics and Computer Science,
College of Science, University of Tehran, Tehran, Iran.}

\email{aminfakhari@ut.ac.ir}

\urladdr{http://math.ipm.ac.ir/$\sim$fakhari/}

\author[S. Yassemi]{S. Yassemi}

\address{S. Yassemi, School of Mathematics, Statistics and Computer Science,
College of Science, University of Tehran, Tehran, Iran.}

\email{yassemi@ut.ac.ir}

\urladdr{http://math.ipm.ac.ir/$\sim$yassemi/}

% ------------------------------------------------------------------------

\begin{abstract}
Let $G$ be a graph with $n$ vertices, $S=\mathbb{K}[x_1,\dots,x_n]$ be the
polynomial ring in $n$ variables over a field $\mathbb{K}$ and $I(G)$ denote the edge ideal of $G$. For every collection $\mathcal{H}$ of connected graphs with $K_2\in \mathcal{H}$, we introduce the notions of $\ind-match_{\mathcal{H}}(G)$ and $\min-match_{\mathcal{H}}(G)$. It will be proved that the inequalities $\ind-match_{\{K_2, C_5\}}(G)\leq{\rm reg}(S/I(G))\leq\min-match_{\{K_2, C_5\}}(G)$ are true. Moreover, we show that if $G$ is a Cohen--Macaulay graph with girth at least five, then ${\rm reg}(S/I(G))=\ind-match_{\{K_2, C_5\}}(G)$. Furthermore, we prove that if $G$ is a paw--free and doubly Cohen--Macaulay graph, then ${\rm reg}(S/I(G))=\ind-match_{\{K_2, C_5\}}(G)$ if and only if every connected component of $G$ is either a complete graph or a $5$-cycle graph. Among other results, we show that for every doubly Cohen--Macaulay simplicial complex, the equality ${\rm reg}(\mathbb{K}[\Delta])={\rm dim}(\mathbb{K}[\Delta])$ holds.
\end{abstract}

% ------------------------------------------------------------------------


% ------------------------------------------------------------------------


% ------------------------------------------------------------------------

\thanks{}

% ------------------------------------------------------------------------

\maketitle

%%%%%%%%%%%%%%%%%%%%%%%%%%%%%%%%%%%%%%%%%%%%%%%%%%%%%%%%%%%%%%%%%%%%%%%%%%

\section{Introduction} \label{sec1}

Let $\mathbb{K}$ be a field and $S = \mathbb{K}[x_1,\ldots,x_n]$  be the
polynomial ring in $n$ variables over $\mathbb{K}$. Suppose that $M$ is a graded $S$-module with minimal free resolution
$$0  \longrightarrow \cdots \longrightarrow  \bigoplus_{j}S(-j)^{\beta_{1,j}(M)} \longrightarrow \bigoplus_{j}S(-j)^{\beta_{0,j}(M)}   \longrightarrow  M \longrightarrow 0.$$
The Castelnuovo--Mumford regularity (or simply, regularity) of $M$, denote by ${\rm reg}(M)$, is defined as follows:
$${\rm reg}(M)=\max\{j-i|\ \beta_{i,j}(M)\neq0\}.$$
The regularity of $M$ is an important invariant in commutative algebra and algebraic geometry.

There is a natural correspondence between quadratic squarefree monomial ideals of $S$ and finite simple graphs with $n$ vertices. To every simple graph $G$ with vertex set $V(G)=\{v_1, \ldots, v_n\}$ and edge set $E(G)$, we associate its {\it edge ideal} $I=I(G)$ defined by
$$I(G)=\big(x_ix_j: v_iv_j\in E(G)\big)\subseteq S.$$The quotient $S/I(G)$ is the {\it edge ring} of $G$. Computing and finding bounds for the regularity of edge ideals have been studied by a number of researchers (see for example \cite{bc}, \cite{dhs}, \cite{ha}, \cite{hv}, \cite{hhkt}, \cite{k}, \cite{khm} and \cite{wo}).

Let $G$ be a graph. A subset $M\subseteq E(G)$ is a {\it matching} if $e\cap e'=\emptyset$, for every pair of edges $e, e'\in M$. The cardinality of the largest matching of $G$ is called the {\it matching number} of $G$ and is denoted by ${\rm match}(G)$. The minimum cardinality of the maximal matchings of $G$ is the {\it minimum matching number} of $G$ and is denoted by $\min-match(G)$. A matching $M$ of $G$ is an {\it induced matching} of $G$ if for every pair of edges $e, e'\in M$, there is no edge $f\in E(G)\setminus M$ with $f\subset e\cup e'$. The cardinality of the largest induced matching of $G$ is called the {\it induced  matching number} of $G$ and is denoted by $\ind-match(G)$.

By \cite{k} and \cite{wo}, we know that for every graph $G$,
\[
\begin{array}{rl}
\ind-match(G)\leq {\rm reg}(S/I(G))\leq \min-match(G).
\end{array} \tag{$\ast$} \label{ast}
\]
The first goal of this paper is to improve  the above inequalities. Indeed, Let $\mathcal{H}$ be a collection of connected graphs with $K_2\in \mathcal{H}$. In Definition \ref{newdef}, for every graph $G$, we introduce the notions of $\ind-match_{\mathcal{H}}(G)$ and $\min-match_{\mathcal{H}}(G)$. It turns out that $\ind-match_{\mathcal{H}}(G)$ is an upper bound for the induced matching number of $G$ and $\min-match_{\mathcal{H}}(G)$ is a lower bound for minimum matching number of $G$. Our first result, Theorem \ref{main1} is less or more known. However, we formulate it and provide its proof for the sake of completeness. It states that if every graph $H\in \mathcal{H}$ has the property that ${\rm reg}(S/I(H))={\rm match}(H)$, then $\ind-match_{\mathcal{H}}(G)\leq {\rm reg}(S/I(G))$, for every graph $G$. Next, in Theorem \ref{main2} we provide an upper bound for the regularity of $S/I(G)$. It will be proven that if every member of $\mathcal{H}$ other than $K_2$ is a cycle graph, then for every graph $G$ we have ${\rm reg}(S/I(G))\leq \min-match_{\mathcal{H}}(G)$. Of particular interest is the case $\mathcal{H}=\{K_2, C_5\}$. As a consequence of Theorems \ref{main1} and \ref{main2}, we conclude that$$\ind-match_{\{K_2, C_5\}}(G)\leq{\rm reg}(S/I(G))\leq\min-match_{\{K_2, C_5\}}(G),$$for every graph $G$ (see Corollary \ref{k5}). Notice that the above inequalities are improvement of inequalities (\ref{ast}). We then concentrate on the left hand side inequality. It will be shown in Theorem \ref{girth5} that for every Cohen--Macaulay graph with girth at least five the equality ${\rm reg}(S/I(G))=\ind-match_{\{K_2, C_5\}}(G)$ holds. Our proof of Theorem \ref{girth5} is based on the characterization of Cohen--Macaulay graphs with girth at least five, given in \cite{hmt}, which states that every such graph belongs to the class $\mathcal{PC}$, i.e., its vertex set admits a particular kind of interesting partition (see Definition \ref{defpc} for more details).

In Section \ref{sec5}, we strengthen the inequality $\ind-match_{\{K_2, C_5\}}(G)\leq{\rm reg}(S/I(G))$ in a special case. More precisely, we consider the paw--free graphs (see Definition \ref{defpaw}) and prove in Theorem \ref{2cmpaw} that for every connected paw--free doubly Cohen--Macaulay graph $G$ other than $C_5$ and complete graphs, we have$$\ind-match_{\{K_2, C_5\}}(G)+1\leq{\rm reg}(S/I(G)).$$ To prove Theorem \ref{2cmpaw}, we first show in Theorem \ref{2cm} that the regularity of the Stanley--Reisner ring a doubly Cohen--Macaulay simplicial complex is equal to its dimension. This result looks interesting by itself and we believe that it can be useful in the study of doubly Cohen--Macaulay simplicial complexes. As a consequence of Theorem \ref{2cmpaw}, we show that a paw--free graph $G$ with ${\rm reg}(S/I(G))=\ind-match_{\{K_2, C_5\}}(G)$ is Gorenstein if and only if its connected components are isomorphic to $K_1, K_2$ or $C_5$ (see Corollary \ref{gorpaw}). This result can be seen as a generalization of the following know result: $K_1,K_2$ and $C_5$ are the only connected Gorenstein graphs with girth at least five.

%%%%%%%%%%%%%%%%%%%%%%%%%%%%%%%%%%%%%%%%%%%%%%%%%%%%%%%%%%%%%%%%%%%%%%%%%%

\section{Preliminaries} \label{sec2}

In this section, we provide the definitions and basic facts which will be used in the next sections.

Let $G$ be a simple graph with vertex set $V(G)=\big\{v_1, \ldots,
v_n\big\}$ and edge set $E(G)$. For a vertex $v_i$, the {\it neighbor set} of $v_i$ is $N_G(v_i)=\{v_j\mid v_iv_j\in E(G)\}$ and we set $N_G[v_i]=N_G(v_i)\cup \{v_i\}$ and call it the {\it closed neighborhood} of $v_i$. The cardinality of $N_G(v_i)$ is called the {\it degree} of $v_i$.  For every subset $U\subset V(G)$, the graph $G\setminus U$ has vertex set $V(G\setminus U)=V(G)\setminus U$ and edge set $E(G\setminus U)=\{e\in E(G)\mid e\cap U=\emptyset\}$. A subgraph $H$ of $G$ is called {\it induced} provided that two vertices of $H$ are adjacent if and only if they are adjacent in $G$. A cycle graph with $n$ vertices is called an {\it $n$-cycle} graph an is denoted by $C_n$. The {\it girth} of $G$, denoted by ${\rm girth}(G)$ is the length of the shortest cycle in $G$. A subset $A$ of $V(G)$ is called an {\it independent subset} of $G$ if there are no edges among the vertices of $A$. The graph $G$ is {\it unmixed} if all its maximal independent sets have the same cardinality. The cardinality of the largest independent subset of vertices of $G$ is called the {\it independence number} of $G$ and is denoted by $\alpha(G)$. A subset $C$ of $V(G)$ is a {\it vertex cover} of the graph $G$ if every edge of $G$ is incident to at least one vertex of $C$. Adding a {\it whisker} to $G$ at a vertex $v_i$ means adding a new vertex $u$ and the edge $uv_i$ to $G$. The graph which is obtained from $G$ by adding a whisker to all of its vertices is denoted by $W(G)$.

A {\it simplicial complex} $\Delta$ on the set of vertices $V(\Delta)=\{v_1,
\ldots,v_n\}$ is a collection of subsets of $\Delta$ which is closed under
taking subsets; that is, if $F \in \Delta$ and $F'\subseteq F$, then also
$F'\in\Delta$. Every element $F\in\Delta$ is called a {\it face} of
$\Delta$, the {\it dimension} of a face $F$ is defined to be $|F|-1$. The {\it dimension} of
$\Delta$ which is denoted by $\dim\Delta$, is defined to be $d-1$, where $d
=\max\{|F|\mid F\in\Delta\}$. The {\it link}
of $\Delta$ with respect to a vertex $v \in V(\Delta)$ is the simplicial complex$${\rm lk_{\Delta}}v=\big\{G
\subseteq V(\Delta)\setminus \{v\}\mid G\cup \{v\}\in \Delta\big\}$$and the {\it deletion} of
$v$ is the simplicial complex$${\rm del_{\Delta}}v=\{G \in \Delta\mid v\notin G\}.$$Let $\Delta_1$ and $\Delta_2$ be two simplicial complexes with disjoint set of vertices. The {\it join} of $\Delta_1$ and $\Delta_2$ is defined to be$$\Delta_1\ast\Delta_2=\{F_1\cup F_2\mid F_1\in \Delta_1, F_2\in \Delta_2\}.$$

Let $\Delta$ be a simplicial complex with vertex set $V(\Delta)=\{v_1,
\ldots,v_n\}$. For every subset $F\subseteq V(\Delta)$, we set ${\it x}_F=\prod_{v_i\in F}x_i$. The {\it
Stanley--Reisner ideal of $\Delta$ over $\mathbb{K}$} is the ideal $I_{
\Delta}$ of $S$ which is generated by those squarefree monomials $x_F$ with
$F\notin\Delta$. The {\it Stanley--Reisner ring of $\Delta$ over $\mathbb{K}$} is defined to be $\mathbb{K}[\Delta]=S/I_{\Delta}$. It is not difficult to check that $\dim\mathbb{K}[\Delta]=\dim\Delta+1$. A simplicial complex is said to be {\it Cohen--Macaulay} (resp. {\it Gorenstein}) if its Stanley--Reisner ring is Cohen--Macaulay (resp. Gorenstein). Also, a simplicial complex $\Delta$ is {\it doubly Cohen--Macaulay} if $\Delta$ is Cohen--Macaulay and for every vertex $v$ of $\Delta$, the simplicial complex ${\rm del_{\Delta}}v$ is Cohen--Macaulay of the same dimension as $\Delta$.

Let $G$ be a graph. The {\it independence simplicial complex} of $G$ is defined by
$$\Delta(G)=\{A\subseteq V(G)\mid A \,\, \mbox{is an independent set in}\,\,
G\}.$$Note that $\dim\Delta(G)=\alpha(G)-1$ and the Stanley--Reisner ideal of $\Delta(G)$ is the same as the edge ideal of $G$. One can easily check that $\Delta(G)$ is the join of the independence simplicial complex of connected components of $G$. A graph is said to be {\it Cohen--Macaulay}, {\it Gorenstein} or {\it doubly Cohen--Macaulay} if its independence simplicial complex satisfies the same property.

Let $\Delta$ be a simplicial complex with $\dim\Delta=d-1$. The number of faces of
$\Delta$ of dimension $i$ is denoted by $f_i$. The sequence $f(\Delta)=(f_0,f_1,\ldots,
f_{d-1})$ is called the {\it $f$-vector} of $\Delta$. Letting $f_{-1}=1$, the {\it $h$-vector}
$h(\Delta)=(h_0,h_1,\ldots,h_d)$ of $\Delta$ is defined by the formula $$\sum_{i=0}^
{d}f_{i-1}(t-1)^{d-i}=\sum_{i=0}^d h_it^ {d-i}.$$

We close this section by reminding the concept of the Hilbert series. Assume that $M=\oplus_{i\in \mathbb{Z}}M_i$ is a graded $S$-module. The {\it Hilbert series} of $M$ is defined to be$${\rm Hilb}_M(z)=\sum_{i\in \mathbb{Z}}({\rm dim}_{\mathbb{K}}M_i)z^i.$$It is well-known that for a $(d-1)$-dimensional simplicial complex $\Delta$, with $h$-vector $h(\Delta)=(h_0,
h_1,\ldots,h_d)$, we have$${\rm Hilb}_{\mathbb{K}[\Delta]}(z)=\frac{h_dz^d+h_{d-1}z^{d-1}+\ldots +h_1z+h_0}{(1-z)^d}.$$

%%%%%%%%%%%%%%%%%%%%%%%%%%%%%%%%%%%%%%%%%%%%%%%%%%%%%%%%%%%%%%%%%%%%%%%%%%

\section{Bounds for the regularity of edge ideals} \label{sec3}

In this section, we provide a lower bound and an upper bound for the regularity of the edge ring of graphs. We first need the following definition which has a central role in our results.

\begin{dfn} \label{newdef}
Let $G$ be a graph with at least one edge and let $\mathcal{H}$ be a collection of connected graphs with $K_2\in \mathcal{H}$. We say that a subgraph $H$ of $G$, is an {\it $\mathcal{H}$--subgraph} if every connected component of $H$ belongs to $\mathcal{H}$. If moreover $H$ is an induced subgraph of $G$, then we say that it is an {\it induced $\mathcal{H}$--subgraph} of $G$. Since $K_2\in \mathcal{H}$, every graph with at least one edge has an induced $\mathcal{H}$--subgraph. An $\mathcal{H}$--subgraph $H$ of $G$ is called {\it maximal} if $G\setminus V(H)$ has no $\mathcal{H}$--subgraph. We set$$\ind-match_{\mathcal{H}}(G):={\rm max} \big\{{\rm match}(H)\mid H \ {\rm is \ an \ induced} \ \mathcal{H}\textendash {\rm subgraph \ of} \ G \big\},$$and$$\min-match_{\mathcal{H}}(G):={\rm min} \big\{{\rm match}(H)\mid H \ {\rm is \ a \ maximal} \ \mathcal{H}\textendash {\rm subgraph \ of} \ G \big\},$$and call them the {\it induced $\mathcal{H}$--matching number} and the {\it minimum $\mathcal{H}$--matching number} of $G$, respectively. We set $\ind-match_{\mathcal{H}}(G)=\min-match_{\mathcal{H}}(G)=0$, when $G$ has no edge.
\end{dfn}

\begin{rem}
It is clear that for $\mathcal{H}=\{K_2\}$, every (induced) $\mathcal{H}$--subgraph of $G$ is indeed a (induced) matching of $G$. Thus, $$\ind-match_{\{K_2\}}(G)=\ind-match(G) \ \ \ {\rm and} \ \ \ \min-match_{\{K_2\}}(G)=\min-match(G).$$We could also define the notion of ${\rm match}_{\mathcal{H}}(G)$ to be the maximum matching number of $\mathcal{H}$--subgraphs of $G$. But it follows from $K_2\in \mathcal{H}$ that ${\rm match}_{\mathcal{H}}(G)$ would be equal to the usual matching number of $G$.
\end{rem}

Let $\mathcal{H}$ be a collection of connected graphs with $K_2\in \mathcal{H}$. It is obvious that for every graph $G$ we have $$\ind-match(G)\leq \ind-match_{\mathcal{H}}(G) \ \ \ \ {\rm and} \ \ \ \ \min-match_{\mathcal{H}}(G)\leq \min-match(G).$$The following examples show that the above inequalities can be strict.

\begin{exmp}
Let $G$ be a connected graph with $\ind-match(G)< {\rm match}(G)$ (for example consider a cycle of length at least $4$) and set $\mathcal{H}=\{K_2, G\}$. Then clearly we have $\ind-match_{\mathcal{H}}(G)={\rm match}(G)$ and thus, $\ind-match(G)< \ind-match_{\mathcal{H}}(G)$.
\end{exmp}

\begin{exmp}
Let $k\geq 1$ be a positive integer and set $\mathcal{H}=\{K_2, C_{2k+1}\}$. Consider the whiskered graph $G=W(C_{2k+1})$. As $C_{2k+1}$ is an induced $\mathcal{H}$--subgraph of $G$, we conclude that $k=\min-match_{\mathcal{H}}(G)< \min-match(G)=k+1.$
\end{exmp}

In the following theorem, we determine a lower bound for the regularity of the edge ring of graphs. It in fact improves the inequality $\ind-match(G)\leq {\rm reg}(S/I(G))$. We believe that this result is less or more known. However, we formulate it and provide its proof for the sake of completeness.

\begin{thm} \label{main1}
Let $\mathcal{H}$ be a subset of$$\mathcal{G}=\{H \mid H \ {\rm is \ a connected \ graph \ with} \ {\rm reg}(S/I(H))={\rm match}(H)\},$$with $K_2\in \mathcal{H}$. Then for every graph $G$ we have$$\ind-match_{\mathcal{H}}(G)\leq {\rm reg}(S/I(G)).$$In particular $\ind-match_{\{K_2, C_5\}}(G)\leq {\rm reg}(S/I(G)).$
\end{thm}

\begin{proof}
Assume that $H$ is an induced $\mathcal{H}$--subgraph of $G$ with $\ind-match_{\mathcal{H}}(G)={\rm match}(H)$. Suppose that $H_1, \ldots, H_t\in \mathcal{H}$ are the connected components of $H$. Then$${\rm reg}(S/I(H))=\sum_{i=1}^t{\rm reg}(S/I(H_i))=\sum_{i=1}^t{\rm match}(H_i)={\rm match}(H),$$where the first equality follows from the fact that the minimal free resolution of $S/I(H)$ is the tensor product of the minimal free resolution of $S/I(H_i)$'s. Hence, we conclude from \cite[Lemma 2.5]{hw} that$$\ind-match_{\mathcal{H}}(G)={\rm match}(H)={\rm reg}(S/I(H))\leq {\rm reg}(S/I(G)).$$

The last part of the theorem is an immediate consequence of the first part.
\end{proof}

Our next goal is to provide an upper bound for the regularity of the edge ring of graphs. This is the content of Theorem \ref{main2}, whose proof is based on the following lemma.

\begin{lem} \label{cycle}
Let $k\geq 3$ be a positive integer. Assume that $G$ is a graph and $H\cong C_k$ is a subgraph of $G$. If $V(H)$ is a vertex cover of $G$, then$${\rm reg}(S/I(G))\leq \lfloor\frac{k}{2}\rfloor.$$
\end{lem}

\begin{proof}
Assume that $V(H)=\{v_1, v_2, \ldots, v_k\}$ and $E(H)=\{v_1v_2, v_2v_3, \ldots, v_{k-1}v_k, v_kv_1\}$. If $k$ is even, then the assumption implies that the set of edges $\{v_1v_2, v_3v_4, \ldots, v_{k-1}v_k\}$ is a maximal matching of $G$. Therefore by the second inequality in (\ref{ast}), we conclude that$${\rm reg}(S/I(G))\leq \frac{k}{2}.$$

Thus, assume that $k$ is odd. Consider the graph $G\setminus N_G[v_1]$. It is clear that $G\setminus N_G[v_1]$ is an induced subgraph of the graph $G'=G\setminus \{v_1, v_2, v_k\}$. One can easily see that $\{v_3, \ldots, v_{k-1}\}$ is a vertex cover of $G'$. Thus, the set of edges $\{v_3v_4, v_5v_6, \ldots, v_{k-2}v_{k-1}\}$ is a maximal matching of $G'$. It follows from \cite[Lemma 3.1]{ha} and the second inequality in (\ref{ast}) that$${\rm reg}(S/I(G\setminus N_G[v_1]))\leq {\rm reg}(S/I(G'))\leq \frac{k-3}{2}.$$

We now consider the graph $G\setminus v_1$. It is clear that $\{v_2, \ldots, v_k\}$ is a vertex cover of $G\setminus v_1$. Thus, the set of edges $\{v_2v_3, v_4v_5, \ldots, v_{k-1}v_k\}$ is a maximal matching of  $G\setminus v_1$ and the second inequality in (\ref{ast}) implies that ${\rm reg}(S/I(G\setminus v_1))\leq \frac{k-1}{2}$.

It now follows from \cite[Lemma 3.2]{dhs} that$${\rm reg}(S/I(G))\leq \max\{{\rm reg}(S/I(G\setminus N_G[v_1]))+1, {\rm reg}(S/I(G\setminus v_1))\}\leq \frac{k-1}{2}=\lfloor\frac{k}{2}\rfloor.$$
\end{proof}

We are now ready to prove the following theorem. It improves the inequality ${\rm reg}(S/I(G))\leq \min-match(G)$.

\begin{thm} \label{main2}
Let $\mathcal{H}$ be a subset of$$\mathcal{G}=\{C_k \mid k\geq 3\}\cup\{K_2\},$$with $K_2\in \mathcal{H}$. Then for every graph $G$ we have$${\rm reg}(S/I(G))\leq \min-match_{\mathcal{H}}(G).$$In particular, ${\rm reg}(S/I(G))\leq\min-match_{\{K_2, C_5\}}(G)$.
\end{thm}

\begin{proof}
Assume that $H$ is a maximal $\mathcal{H}$--subgraph of $G$ with $\min-match_{\mathcal{H}}(G)={\rm match}(H)$. Suppose that $H_1, \ldots, H_t\in \mathcal{H}$ are the connected components of $H$. Since $K_2\in \mathcal{H}$ and $H$ is a maximal $\mathcal{H}$--subgraph of $G$, we conclude that $V(G)\setminus V(H)$ is an independent subset of vertices of $G$. In other words, $V(H)$ is a vertex cover of $G$. For every integer $i$ with $1\leq i\leq t$, let $G_i$ be the graph with vertex set $V(G_i)=V(G)$ and edge set$$E(G_i)=\{e\in E(G)\mid \ {\rm at \ least \ one \ end \ point \ of } \ e \ {\rm belongs \ to} \ V(H_i)\}.$$Then $I(G)=\sum_{i=1}^tI(G_i)$ and it follows from \cite[Corollary 3.2]{he} (see also \cite[Theorem 1.2]{km}) that$${\rm reg}(S/I(G))\leq \sum_{i=1}^t{\rm reg}(S/I(G_i)).$$Thus we only need to prove that for every integer $i$ with $1\leq i\leq t$, the inequality ${\rm reg}(S/I(G_i))\leq {\rm match}(H_i)$ holds.

It is clear that for every $i$ with $1\leq i\leq t$, the vertex set of $H_i$ is a vertex cover of $G_i$. Hence, if $H_i$ is a cycle, then Lemma \ref{cycle} shows that$${\rm reg}(S/I(G_i))\leq \Big\lfloor\frac{\mid V(H_i)\mid}{2}\Big\rfloor={\rm match}(H_i).$$If $H_i=K_2$, then the single edge of $H_i$ forms a maximal matching for $G_i$. Thus$${\rm reg}(S/I(G_i))\leq \min-match(G_i)=1= {\rm match}(H_i),$$and the proof is complete.

The last part of the theorem is an immediate consequence of the first part.
\end{proof}

We close this section by summarizing the particular cases of Theorems \ref{main1} and \ref{main2}.

\begin{cor} \label{k5}
For every graph $G$, we have$$\ind-match_{\{K_2, C_5\}}(G)\leq{\rm reg}(S/I(G))\leq\min-match_{\{K_2, C_5\}}(G).$$
\end{cor}

%%%%%%%%%%%%%%%%%%%%%%%%%%%%%%%%%%%%%%%%%%%%%%%%%%%%%%%%%%%%%%%%%%%%%%%%%%

\section{Cohen--Macaulay graphs with girth at least five} \label{sec4}

We proved in Theorem \ref{main1} that for every graph $G$ we have $\ind-match_{\{K_2, C_5\}}(G)\leq{\rm reg}(S/I(G))$. A natural question is to determine the graphs for which the equality ${\rm reg}(S/I(G))=\ind-match_{\{K_2, C_5\}}(G)$ holds. There are several classes of graphs with ${\rm reg}(S/I(G))=\ind-match(G)$. They include\\

$\bullet$ chordal graphs (see \cite[Corollary 6.9]{hv}),

$\bullet$ sequentially Cohen--Macaulay bipartite graphs (see \cite[Theorem 3.3]{v}),

$\bullet$ unmixed bipartite graphs (see \cite[Theorem 1.1]{ku}),

$\bullet$ very well--covered graphs (see \cite[Theorem 1.3]{mmcrty}),

$\bullet$ vertex decomposable graphs which have no $C_5$--subgraph (see \cite[Theorem 2.4]{khm}).\\

As $\ind-match(G)$ is a lower bound for $\ind-match_{\{K_2, C_5\}}(G)$, Corollary \ref{k5} implies that for every graph $G$ in the above classes, ${\rm reg}(S/I(G))=\ind-match_{\{K_2, C_5\}}(G)$. In Theorem \ref{girth5}, we determine a new class of graphs, namely Cohen--Macaulay graphs with girth at least five, for which the equality ${\rm reg}(S/I(G))=\ind-match_{\{K_2, C_5\}}(G)$ holds. The proof of Theorem \ref{girth5} is based on the characterization of Cohen--Macaulay graphs with girth at least five, given in \cite{hmt}.

Let $G$ be a graph. A $5$-cycle of $G$ is said to be {\it basic} if it does not contain two adjacent vertices
of degree three or more in $G$. An edge of $G$ which is incident to a vertex of degree $1$ is called a {\it pendant} edge.
Let $C(G)$ denote the set of all vertices which belong to basic $5$-cycles and let $P(G)$
denote the set of vertices which are incident to pendant edges of $G$.

\begin{dfn} \label{defpc}
A graph $G$ is said to belong to the class $\mathcal{PC}$ if
\begin{itemize}
\item[(1)] $V(G)$ can be partitioned as $V(G)=P(G)\cup C(G)$ and
\item[(2)] the pendant edges form a perfect matching of $P(G)$ and
\item[(3)] the vertices of basic $5$-cycles form a partition of $C(G)$.
\end{itemize}
\end{dfn}

\begin{prop} \label{pc}
For every graph $G$ in the class $\mathcal{PC}$, we have$${\rm reg}(S/I(G))=\ind-match_{\{K_2, C_5\}}(G).$$
\end{prop}

\begin{proof}
Assume that $G_1, \ldots, G_t$ are the connected components of $G$. As the minimal free resolution of $S/I(G)$ is the tensor product of the minimal free resolution of $S/I(G_i)$'s, it follows that ${\rm reg}(S/I(G))=\sum_{i=1}^t{\rm reg}(S/I(G_i))$. On the other hand,$$\ind-match_{\{K_2, C_5\}}(G)=\sum_{i=1}^t\ind-match_{\{K_2, C_5\}}(G_i).$$Thus it is sufficient to prove the theorem for connected graphs.

Assume that $G$ is a connected graph in the class $\mathcal{PC}$. We use induction on the number of edges of $G$. If $G$ has one edge, then $G\cong K_2$ and there is nothing to prove. Thus, assume that $G$ has at least two edges. The vertex set of $G$ can be partitioned as $V(G)=P(G)\cup C(G)$, where $C(G)$ is the set of all vertices which belong to basic $5$-cycles and $P(G)$ is the set of vertices which are incident to pendant edges of $G$. If $C(G)=\emptyset$, then $V(G)=P(G)$. As the pendant edges form a perfect matching of $P(G)$, it follows that $G$ is a whiskered graph and by \cite[Theorem 13]{bv} and Corollary \ref{k5}, we have$$\ind-match_{\{K_2, C_5\}}(G)\leq{\rm reg}(S/I(G))=\ind-match(G)\leq \ind-match_{\{K_2, C_5\}}(G)$$and hence ${\rm reg}(S/I(G))=\ind-match_{\{K_2, C_5\}}(G)$.

Therefore suppose that $C(G)\neq\emptyset$ and let $L\cong C_5$ be a basic $5$-cycle of $G$. Assume that $V(L)=\{x, y ,z ,u, v\}$ is the vertex set of $L$ and $E(L)=\{xy, yz, zu, uv, vx\}$ is its edge set. If $G=L$, then ${\rm reg}(S/I(G))=2=\ind-match_{\{K_2, C_5\}}(G)$ and there is nothing to prove. Thus assume that $G\neq L$. Since $G$ is connected, $L$ has a vertex of degree at least three in $G$. Without loss of generality, we may assume that ${\rm deg}_G(x)\geq 3$. Set $G'=G\setminus x$. As $L$ is a basic $5$-cycles, ${\rm deg}_G(v)={\rm deg}_G(y)=2$. Therefore the edges $yz$ and $uv$ are pendant edges in $G'$. This shows that the vertex set of $G'$ can be partitioned as $V(G')=P(G')\cup C(G')$, where$$P(G')=P(G)\cup\{u,v, y,z\}$$and $C(G')=C(G)\setminus V(L)$. Hence $G'$ belongs to the class $\mathcal{PC}$ and by induction hypothesis, we conclude that$${\rm reg}(S/I(G'))=\ind-match_{\{K_2, C_5\}}(G')\leq \ind-match_{\{K_2, C_5\}}(G),$$where the inequality follows from the fact that $G'$ is an induced subgraph of $G$.

We now set $G''=G\setminus N_G[x]$. By the definition of basic $5$-cycles, at least one of the vertices $u$ and $z$ has degree two in $G$. Without loss of generality, assume that ${\rm deg}_G(u)=2$. It follows that ${\rm deg}_{G''}(u)\leq 1$.\\

{\bf Claim 1.} The graph obtained from $G''$ by deleting the isolated vertices belongs to the class $\mathcal{PC}$.

{\it Proof of Claim 1.} Let $C\in C(G)$ be a basic $5$-cycle of $G$. In order to prove Claim 1, it is sufficient to prove that either $C\setminus N_G[x]$ is a basic $5$-cycle of $G''$ or its non-isolated vertices can be partitioned to pendant edges of $G''$, where the pendant edges form a perfect matching for non-isolated vertices of $C\setminus N_G[x]$. As $C$ has at most two vertices of degree at least three, we have $\mid V(C)\cap N_G[x]\mid\leq 2$, if $C\neq L$. We consider the following cases.

{\bf Case 1.} Assume that $C=L$. If $xz\in E(G)$, then $L\setminus N_G[x]$ has only one vertex, namely $u$ which is an isolated vertex in $G''$. Thus, assume that $xz\notin E(G)$. Then $L\setminus N_G[x]=uz$, which is a pendant edge of $G''$.

{\bf Case 2.} If $V(C)\cap N_G[x]=\emptyset$, then $C\setminus N_G[x]=C$, which is a basic $5$-cycle.

{\bf Case 3.} Assume that $\mid V(C) \cap N_G[x]\mid=1$. Let $V(C)=\{w_1, w_2, w_3, w_4, w_5\}$ be the vertex set of $C$ and $E(C)=\{w_1w_2, w_2w_3, w_3w_4, w_4w_5, w_1w_5\}$ be its edge set. Without loss of generality, assume that $V(C)\cap N_G[x]=\{w_1\}$. Then it follows from the definition of basic $5$-cycles that $w_2$ and $w_5$ have degree one in $G''$. Thus, the edges $w_2w_3$ and $w_4w_5$ are the pendant edges of $G''$ which form a perfect matching for $C\setminus N_G[x]$.

{\bf Case 4.} Assume that $\mid V(C)\cap N_G[x]\mid=2$. Similar to Case 3, suppose that $V(C)=\{w_1, w_2, w_3, w_4, w_5\}$ is the vertex set of $C$ and $E(C)=\{w_1w_2, w_2w_3, w_3w_4, w_4w_5, w_1w_5\}$ is its edge set. Since $C$ has no adjacent vertices of degree at least three, we may assume that $V(C)\cap N_G[x]=\{w_1, w_3\}$. As ${\rm deg}_G(w_2)=2$, it follows that $w_2$ is an isolated vertex of $G''$. On the other hand ${\rm deg}_G(w_4)={\rm deg}_G(w_5)=2$, which implies that the edge $w_4w_5$ is a pendant edge in $G''$, which forms a perfect matching for the graph obtained  from $C\setminus N_G[x]$ by ignoring its isolated vertices.

The proof of Claim 1 is thus complete.\\

{\bf Claim 2.} $\ind-match_{\{K_2, C_5\}}(G'')\leq \ind-match_{\{K_2, C_5\}}(G)-1$.

{\it Proof of Claim 2.} Let $H$ be an induced $\{K_2, C_5\}$-subgraph of $G''$ with$$\ind-match_{\{K_2, C_5\}}(G'')={\rm match}(H).$$Assume that $H_1, \ldots, H_{\ell}$ are the connected components of $H$. By relabeling the $H_i$'s (if necessary), we may assume that there exists an integer $0\leq s\leq \ell$ such that $H_1, \ldots, H_s$ are isomorphic to $K_2$ and $H_{s+1}, \ldots, H_{\ell}$ are isomorphic to $C_5$. Since ${\rm deg}_{G''}(u)\leq 1$, the edge $uz$ does not belong to $\cup_{i=s+1}^{\ell}E(H_i)$. Remind that the degree of $v$ and $y$ in $G$ is equal to two. If $uz=H_i$, for some $1\leq i\leq s$, then$$H_1\sqcup\ldots \sqcup H_{i-1}\sqcup H_{i+1}\sqcup\ldots \sqcup H_s\sqcup H_{s+1}\ldots \sqcup H_{\ell}\sqcup L$$is an induced $\{K_2, C_5\}$-subgraph of $G$ whose matching number is one more than the matching number of $H$. This implies the claimed inequality. Therefore, assume that $uz$ does not belong to $\cup_{i=1}^{\ell}E(H_i)$. This together with ${\rm deg}_{G''}(u)\leq 1$, implies that $u\notin V(H)$. Since $N_G(v)=\{u,x\}$, we conclude that the (disjoint) union of $H$ and the edge $vx$ is an induced $\{K_2, C_5\}$-subgraph of $G$ whose matching number is one more than the matching number of $H$. This completes the proof of Claim 2.\\

Since the isolated vertices have no effect on the edge ideal, it follows from Claims 1, 2 and the induction hypothesis that$${\rm reg}(S/I(G''))=\ind-match_{\{K_2, C_5\}}(G'')\leq \ind-match_{\{K_2, C_5\}}(G)-1.$$

Finally by \cite[Lemma 3.2]{dhs}, we conclude that$${\rm reg}(S/I(G))\leq \max\{{\rm reg}(S/I(G''))+1, {\rm reg}(S/I(G'))\}\leq \ind-match_{\{K_2, C_5\}}(G),$$which together with Corollary \ref{k5} implies that ${\rm reg}(S/I(G))=\ind-match_{\{K_2, C_5\}}(G)$.
\end{proof}

We are now ready to prove the main result of this section.

\begin{thm} \label{girth5}
Let $G$ be a Cohen--Macaulay graph with ${\rm girth}(G)\geq 5$. Then$${\rm reg}(S/I(G))=\ind-match_{\{K_2, C_5\}}(G).$$
\end{thm}

\begin{proof}
Without loss of generality, we may assume that $G$ has no isolated vertex. Let $G_1, \ldots, G_m$ be the connected components of $G$. By \cite{f2}, we know that a graph is Cohen--Macaulay if and only if its connected components are Cohen--Macaulay. Thus, $G_1, \ldots, G_m$ are Cohen--Macaulay graphs and it follows from \cite[Theorem 2.4]{hmt} that they belong to the class $\mathcal{PC}$. Hence by Proposition \ref{pc}, we have$${\rm reg}(S/I(G))=\sum_{i=1}^m{\rm reg}(S/I(G_i))=\sum_{i=1}^m\ind-match_{\{K_2, C_5\}}(G_i)=\ind-match_{\{K_2, C_5\}}(G),$$where the first equality follows from the fact that the minimal free resolution of $S/I(G)$ is the tensor product of the minimal free resolution of $S/I(G_i)$'s.
\end{proof}

%%%%%%%%%%%%%%%%%%%%%%%%%%%%%%%%%%%%%%%%%%%%%%%%%%%%%%%%%%%%%%%%%%%%%%%%%%

\section{Paw--free graphs} \label{sec5}

In this section, we study the regularity of edge ideal of paw--free graphs. These graphs are defined as follows.

\begin{dfn} \label{defpaw}
A graph with vertices $v,x,y,z$ and edges $xy, yz, xz, vx$ is called a {\it paw}. The graph $G$ is said to be {\it paw--free} if it has no paw as an induced subgraph.
\end{dfn}

We prove in Theorem \ref{2cmpaw} that for every connected paw--free doubly Cohen--Macaulay graph other than complete graphs and $5$-cycle graph, we have  $$\ind-match_{\{K_2, C_5\}}(G)+1\leq{\rm reg}(S/I(G)).$$To prove Theorem \ref{2cmpaw}, we first show in Theorem \ref{2cm} that the regularity of the Stanley--Reisner ring of a doubly Cohen--Macaulay simplicial complex is equal to its dimension. This result looks to be interesting by itself. In order to prove Theorem \ref{2cm}, we need the following lemma which states that the entries of the $h$-vector of a doubly Cohen--Macaulay simplicial complex are strictly positive.

\begin{lem} \label{hvec}
Let $\Delta$ be a $(d-1)$-dimensional doubly Cohen--Macaulay simplicial complex with $h(\Delta)=(h_0, h_1, \ldots, h_d)$. Then for every integer $0\leq i\leq d$, we have $h_i\geq 1$.
\end{lem}

\begin{proof}
We use induction on $d$. There is nothing to prove for $d=0$, as $h_0=1$. Thus assume that $d\geq 1$. The equality $h_0=1$ is true for every simplicial complex. Therefore, we must prove that $h_1, \ldots, h_d$ are positive. Suppose that $v$ is an arbitrary vertex of $\Delta$. Set $$h({\rm del}_{\Delta} v)=(h'_0, h'_1, \ldots, h'_d) \ \ \ \ \ \ \ \ {\rm and} \ \ \ \ \ \ \ \ h({\rm lk}_{\Delta} v)=(h''_0, h''_1, \ldots, h''_{d-1}).$$By \cite[Proposition 9.7]{w}, we know that ${\rm lk}_{\Delta} v$ is a $(d-2)$-dimensional doubly Cohen--Macaulay simplicial complex. Therefore, the induction hypothesis implies that the integers $h''_0, h''_1, \ldots, h''_{d-1}$ are strictly positive. On the other hand, ${\rm del}_{\Delta} v$ is a Cohen--Macaulay simplicial complex and it follows  from Macaulay's Theorem \cite[Corollary 4.1.10]{bh} that $h'_0, h'_1, \ldots, h'_d$ are nonnegative integers. Finally, using \cite[Lemma 4.1]{a}, we deduce that for every $i=1, \ldots, d$, $$h_i=h'_i+h''_{i-1}\geq 1.$$
\end{proof}

\begin{thm} \label{2cm}
Let $\Delta$ be a doubly Cohen--Macaulay simplicial complex. Then ${\rm reg}(\mathbb{K}[\Delta])={\rm dim}(\mathbb{K}[\Delta])$.
\end{thm}

\begin{proof}
Without loss of generality, we may assume that $\mathbb{K}$ is an infinite field. Suppose that ${\rm dim} \Delta=d-1$ and $h(\Delta)=(h_0, h_1, \ldots, h_d)$ is the $h$-vector of $\Delta$. Therefore the Hilbert series of $\mathbb{K}[\Delta]$ is$${\rm Hilb}_{\mathbb{K}[\Delta]}(z)=\frac{h_dz^d+h_{d-1}z^{d-1}+\ldots + h_1z+h_0}{(1-z)^d}.$$As $\Delta$ is Cohen-Macaulay and $\mathbb{K}$ is infinite, we can choose a $\mathbb{K}[\Delta]$-regular sequence of linear forms ${\bf u}=u_1, u_2, \ldots, u_d$ in $S$. Then$${\rm Hilb}_{\mathbb{K}[\Delta]/({\bf u})\mathbb{K}[\Delta]}(z)=h_dz^d+h_{d-1}z^{d-1}+\ldots + h_1z+h_0.$$By \cite[Theorem 20.2]{p'}, we have ${\rm reg}(\mathbb{K}[\Delta])={\rm reg}(\mathbb{K}[\Delta]/({\bf u})\mathbb{K}[\Delta])$. Since $\mathbb{K}[\Delta]/({\bf u})\mathbb{K}[\Delta]$ is Artinian, \cite[Exercise 20.18]{e} implies that $${\rm reg}(\mathbb{K}[\Delta]/({\bf u})\mathbb{K}[\Delta])=\max \{i \mid h_i\neq 0\}.$$The assertion now follows from Lemma \ref{hvec}.
\end{proof}

Theorem \ref{2cmpaw} is essentially a combination of Theorem \ref{2cm} and the following structural lemma.

\begin{lem} \label{pawf}
Let $G$ be a paw--free graph without isolated vertices and assume that $\alpha(G)=\ind-match_{\{K_2, C_5\}}(G)$. Then every connected component of $G$ is either a complete graph or a $5$-cycle graph.
\end{lem}

\begin{proof}
We use induction on $n:=\mid V(G)\mid$. If $n=2$, then $G\cong K_2$ and there is nothing to prove. Therefore, assume that $G$ is a graph with at least three vertices. Let $H$ be an induced $\{K_2, C_5\}$-subgraph of $G$ with $\ind-match_{\{K_2, C_5\}}(G)={\rm match}(H)$. Since the connected components of $H$ are $K_2$ and $C_5$, we conclude that$$\alpha(H)={\rm match}(H)=\ind-match_{\{K_2, C_5\}}(G)=\alpha(G).$$ If $V(H)=V(G)$, then $G=H$ and the assertion follows. Therefore, suppose that $V(H)\neq V(G)$. Hence, there is a vertex $v\in V(G)\setminus V(H)$. Set $G'=G\setminus v$. As $G'$ is an induced subgraph of $G$ which contains $H$, it follows from the above equalities that $\alpha(G')=\alpha(G)=\alpha(H)$. Also, it is clear that$$\ind-match_{\{K_2, C_5\}}(G')={\rm match}(H).$$Therefore, $\alpha(G')=\ind-match_{\{K_2, C_5\}}(G')$. Assume that $G'$ has an isolated vertex, say $u$. Then $H$ is an induced subgraph of $G'\setminus u$ and therefore, $G'\setminus u$ has an independent set of cardinality $\alpha(H)=\alpha(G)$. This implies that $\alpha(G')\geq \alpha(G)+1$, a contradiction.  Thus $G'$ has no isolated vertex. It then follows from the induction hypothesis that every connected component of $G'$ is either a complete graph or a $5$-cycle graph. Assume that $H_1, \ldots, H_{\ell}$ are the connected components of $G'$. By relabeling the $H_i$'s (if necessary), we may assume that there exists an integer $0\leq s\leq \ell$ such that $H_1, \ldots, H_s$ are isomorphic to $C_5$ and $H_{s+1}, \ldots, H_{\ell}$ are complete graphs. Thus, $\alpha(G)=\alpha(G')=\ell+s$. First assume that there is an integer $i$ with $s+1\leq i\leq \ell$ such that $v$ is adjacent to every vertex of $H_i$. Since $H_i$ has at least two vertices and $G$ is paw--free, $v$ is adjacent to no vertex in $H_1, \ldots, H_{i-1}, H_{i+1}, \ldots, H_{\ell}$. This shows that the connected components are $G$ are $H_1, \ldots, H_{i-1}, H'_i, H_{i+1}, \ldots, H_{\ell}$, where $H'_i$ is the complete graph with vertex set $V(H_i)\cup \{v\}$. Hence, every connected component of $G$ is either a complete graph or a $5$-cycle graph. Thus, assume that for every integer $i$ with $s+1\leq i\leq \ell$, there is a vertex, say $u_i\in V(H_i)$, such that $vu_i\notin E(G)$.

Suppose now that for every integer $1\leq i\leq s$, there are two vertices $w_i, v_i\in V(H_i)$, with $w_iv_i, vw_i, vv_i\notin E(G)$. Then the set$$\{v, v_1, \ldots, v_s, w_1, \ldots, w_s, u_{s+1}, \ldots, u_{\ell}\}$$is an independent subset of vertices of $G$ with cardinality $\ell+s+1$. This contradicts $\alpha(G)=\ell+s$. Thus, there is an integer $1\leq i\leq s$ such that for every non-adjacent vertices $x, y\in V(H_i)$, we have either $vx\in E(G)$ or $vy\in E(G)$. Assume that $V(H_i)=\{z_1, z_2, z_3, z_4, z_5\}$ is the vertex set of $H_i$ and $E(H_i)=\{z_1z_2, z_2z_3, z_3z_4, z_4z_5, z_1z_5\}$ is its edge set. By the above argument at least one of $z_1$ and $z_3$ must be adjacent to $v$. Without lose of generality assume that $vz_1\in E(G)$. Similarly. at least on of $z_2$ and $z_5$ must be adjacent to $v$. Without lose of generality assume that $vz_2\in E(G)$. If $vz_3\notin E(G)$, then the induced subgraph of $G$ on the vertices $z_1, z_2, z_3, v$ is a paw which is a contradiction. Thus, $vz_3\in E(G)$. If $vz_4\notin E(G)$, then the induced subgraph of $G$ on the vertices $z_2, z_3, z_4, v$ is a paw which is again a contradiction. Thus, $vz_4\in E(G)$. But then the induced subgraph of $G$ on the vertices $z_1, z_3, z_4, v$ is a paw. This contradiction completes our proof.
\end{proof}

Let $\Delta_1$ and $\Delta_2$ be two simplicial complexes with disjoint set of variables. Fr${\rm \ddot{o}}$berg \cite{f2} proves that  $\Delta_1\ast\Delta_2$ is Cohen--Macaulay if and only if $\Delta_1$ and $\Delta_2$ are Cohen--Macaulay. Using this result, one can easily prove the following proposition.

\begin{prop} \label{join2cm}
Let $\Delta_1$ and $\Delta_2$ be two simplicial complexes with disjoint set of variables. Then $\Delta_1\ast\Delta_2$ is doubly Cohen--Macaulay if and only if $\Delta_1$ and $\Delta_2$ are doubly Cohen--Macaulay.
\end{prop}

We are now ready to prove the following theorem.

\begin{thm} \label{2cmpaw}
Let $G$ be a paw--free graph. Then the following statements are equivalent.
\begin{itemize}
\item[(i)] $G$ is doubly Cohen--Macaulay and ${\rm reg}(S/I(G))=\ind-match_{\{K_2, C_5\}}(G)$.
\item[(ii)] Every connected component of $G$ is either a complete graph with at least two vertices or a $5$-cycle graph.
\end{itemize}
\end{thm}

\begin{proof}
Doubly Cohen--Macaulay graphs have not isolated vertices. Thus, the implication (i)$\Rightarrow$(ii) follows immediately from Theorem \ref{2cm} and Lemma \ref{pawf}.

For (ii)$\Rightarrow$(i), notice that the complete graphs with at least two vertices and $5$-cycle graph are doubly Cohen--Macaulay. Thus, Proposition \ref{join2cm} implies that $G$ is doubly Cohen--Macaulay. The equality ${\rm reg}(S/I(G))=\ind-match_{\{K_2, C_5\}}(G)$ follows from the following facts
\begin{itemize}
\item[(1)] The regularity of the edge ring of a graph is equal to the sum of the regularity of the edge ring of its connected components.
\item[(2)] The regularity of the edge ring of a complete graph is one.
\item[(3)] The regularity of the edge ring of a $5$-cycle graph is two.
\end{itemize}
\end{proof}

\begin{cor}  \label{gorpaw}
Let $G$ be a paw--free graph. Then the following statements are equivalent.
\begin{itemize}
\item[(i)] $G$ is Gorenstein and ${\rm reg}(S/I(G))=\ind-match_{\{K_2, C_5\}}(G)$.
\item[(ii)] The connected components of $G$ are $K_1, K_2$ or $C_5$.
\end{itemize}
\end{cor}

\begin{proof}
(i)$\Rightarrow$(ii) Let $G'$ be the graph obtained from $G$ by deleting the isolated vertices. We know from \cite[Lemma 1.3]{ht1} that $G'$ is a doubly Cohen--Macaulay graph. Thus it follows from Theorem \ref{2cmpaw} that every connected component of $G'$ is either a complete graph or a $5$-cycle graph. On the other hand, by \cite{f2}, a graph is Gorenstein if and only if its connected components are Gorenstein. One can easily check that the complete graphs with at least three vertices are not Gorenstein. Thus, every connected component of $G'$ is either a complete graph $K_2$ or a $5$-cycle graph.

(ii)$\Rightarrow$(i) The equality ${\rm reg}(S/I(G))=\ind-match_{\{K_2, C_5\}}(G)$ follows immediately from Theorem \ref{2cmpaw}, by noticing that the isolated vertices have no effect on regularity. Since $K_1, K_2$ and $C_5$  are Gorenstein graphs, we conclude from \cite{f2} that their disjoint union is Gorenstein too.
\end{proof}

As an immediate consequence of Theorem \ref{girth5} and Corollary \ref{gorpaw}, we obtain the following know result (see \cite[Page 10]{hmt}).

\begin{cor}
Let $G$ be a graph. Assume that the girth of $G$ is at least $5$. Then $G$ is Gorenstein if and only if $G$ is the disjoint union of copies of $K_1, K_2$ and $C_5$.
\end{cor}

%%%%%%%%%%%%%%%%%%%%%%%%%%%%%%%%%%%%%%%%%%%%%%%%%%%%%%%%%%%%%%%%%%%%%%%%%%

%\section*{Acknowledgment}

%%%%%%%%%%%%%%%%%%%%%%%%%%%%%%%%%%%%%%%%%%%%%%%%%%%%%%%%%%%%%%%%%%%%%%%%%%

%%%%%%%%%%%%%%%%%%%%%%%%%%%%%%%%%%%%%%%%%%%%%%%%%%%%%%%%%%%%%%%%%%%%%%%%%%

\end{document}